\documentclass{amsart}
\usepackage[mathscr]{eucal}
\usepackage{enumerate}

\newtheorem{theorem}{Theorem}[section]
\newtheorem{lemma}[theorem]{Lemma}

\theoremstyle{definition}

\theoremstyle{remark}

\numberwithin{equation}{section}

\begin{document}
	
	\title[ Linear maps on $\mathcal{L}(\ell_{\MakeLowercase{p}}^{\MakeLowercase{n}},\ell_{\MakeLowercase{p}}^{\MakeLowercase{m}}), ~({\MakeLowercase{p}}\in \{1,\infty\})$ preserving parallel pairs] {{Linear maps on $\mathcal{L}(\ell_{\MakeLowercase{p}}^{\MakeLowercase{n}},\ell_{\MakeLowercase{p}}^{\MakeLowercase{m}}),~~({\MakeLowercase{p}}\in \{1,\infty\})$ preserving parallel pairs}}

\author[Arpita Mal]{Arpita Mal}

\address[]{Dhirubhai Ambani University\\ Gandhinagar-382007\\ India.}
\email{arpitamalju@gmail.com}



\subjclass[2020]{Primary 15A86, 15A60, Secondary  47B01, 46B20}
\keywords{Parallelism; parallel pair preserver; triangle equality attaining pair preserver; linear preserver problem}



\date{}
\maketitle
\begin{abstract}
Two vectors $x,y$ of a Banach space are said to form a parallel (resp. triangle equality attaining or TEA) pair if $\|x+\lambda y\|=\|x\|+\|y\|$ holds for some scalar $\lambda$ with $|\lambda|=1$ (resp. $\lambda=1$). For $p\in \{1,\infty\},$ and $ m,n\geq 2,$ we study the linear maps $T:\mathcal{L}(\ell_p^{\MakeLowercase{n}},\ell_p^{\MakeLowercase{m}})\to \mathcal{L}(\ell_p^{\MakeLowercase{n}},\ell_p^{\MakeLowercase{m}})$ that preserve parallel (resp. TEA) pairs, that is, those linear maps $T$ for which $T(A),T(B)$ form a parallel  (resp. TEA) pair whenever $A,B$ form a parallel (resp. TEA) pair of $\mathcal{L}(\ell_p^{\MakeLowercase{n}},\ell_p^{\MakeLowercase{m}}).$ 
We prove that if $T$ is non-zero, then the following are equivalent.
	\begin{enumerate}
	\item $T$ preserves TEA pairs.
	\item $T$ preserves parallel pairs and $rank(T)>1$.
	\item $T$ preserves parallel pairs and $T$ is invertible.
	\item  $T$ is a scalar multiple of an isometry.	
\end{enumerate}
\end{abstract}

\section{Introduction}
Suppose $x,y$ are two vectors of a Banach space $\mathcal{X}$ over the field $\mathbb{F}\in\{\mathbb{R}, \mathbb{C}\}.$ If equality holds in the triangle inequality, that is, if 
$$\|x+y\|=\|x\|+\|y\|$$
 holds, then we say that $x,y$ form a  triangle equality attaining pair, in short TEA pair. On the other hand, if there exists $\lambda\in \mathbb{T}:=\{\mu\in \mathbb{F}:|\mu|=1\}$ such that 
 $$\|x+\lambda y\|=\|x\|+\|y\|,$$
  then we say that $x,y$ are parallel, and denote it by $x\parallel y.$ Clearly, this notion is symmetric and homogeneous, that is, $x\parallel y \Leftrightarrow y\parallel x,$ and $x\parallel y\Rightarrow \alpha x\parallel \beta y  $ for all $\alpha,\beta\in \mathbb{F}.$ If $x\parallel y,$ we also say that $x,y$ form a parallel pair.  A linear map $T:\mathcal{X}\to \mathcal{X}$ preserves parallel pairs if 
\[x\parallel y\quad \Rightarrow \quad  T(x)\parallel T(y) \text{ for all } x,y\in \mathcal{X},\] whereas $T$ preserves TEA pairs, if $T(x),T(y)$ form a  TEA pair whenever $x,y$ form a TEA pair. 
Note that, any two linearly dependent vectors are always parallel. Therefore, if $rank(T)=1,$ then $T$ must preserve parallel pairs.  Moreover, if $T$ preserves TEA pairs, then $T$ also preserves parallel pairs. However, the converse is not true. For example, suppose $T:\ell_1^2\to\ell_1^2$ is defined as $T(a,b)=(-3a+b)(1,0).$ Then since $rank(T)=1,$ $T$ preserves parallel pairs, but $T(0,1),T(1,1)$ do not form a TEA pair, whereas $(0,1),(1,1)$ form a TEA pair. On the other hand, if $T$ is a scalar multiple of an isometry on $\mathcal{X}$, then $T$ clearly preserves parallel pairs as well as TEA pairs. However, there exist linear maps preserving TEA (or parallel) pairs  which are not necessarily scalar multiple of isometries, see \cite{LTWW}. Moreover, linear maps preserving parallel or TEA pairs may vary significantly for real and complex scalars, for example see \cite{LTWW2}. These  changes of such maps on different spaces makes the study more fascinating.
Linear maps preserving parallel (or TEA) pairs have been recently studied by several authors on different Banach spaces. See \cite{KLPS,KLSP,LTWW, LTWW2} for the study on $\ell_p$, 
and on the spaces of matrices equipped with the spectral norm, Ky-Fan $k$-norm, and  $k$-numerical radius.
\\

In this article, we consider $\mathcal{L}(\ell_p^n,\ell_p^m),$ the space of linear operators from $\ell_p^n$ to $\ell_p^m,$ equipped with the usual operator norm, where $p\in \{1,\infty\},$ and study parallel or TEA pairs preserving linear maps on this space.  Since for $m=1,$ the space $\mathcal{L}(\ell_p^n,\ell_p^m)$ is isometrically isomorphic to either  $\ell_1^n$ or $\ell_\infty^n,$ the case is already considered in \cite{LTWW}. Therefore,   we always assume that $n,m>1.$ After some preliminary results in section \ref{sec-pre}, 
 we prove the following theorem in  section \ref{sec-proof}.
\begin{theorem}\label{th-01}
	Suppose $T:\mathcal{L}(\ell_p^n,\ell_p^m)\to \mathcal{L}(\ell_p^n,\ell_p^m)$ is a non-zero linear map, where $p\in \{1,\infty\}$. Then the following are equivalent.\\
	\indent \quad$(1)$ $T$ preserves TEA pairs.\\
	\indent \quad	 $(2)$ $T$ preserves parallel pairs and $rank(T)>1$.\\	 
	\indent \quad	 $(3)$ $T$ preserves parallel pairs and $T$ is invertible.\\ 
	\indent \quad	 $(4)$  $T$ is a scalar multiple of an isometry.	
\end{theorem}
The proof of the theorem is quiet involved and technical. Note that, $(4)\Rightarrow (1)$ is trivial. In subsections \ref{sub-12}, \ref{sub-23}, and \ref{sub-34}, we  prove the implications $(1) \Rightarrow (2),$  $(2) \Rightarrow (3),$ and $(3)\Rightarrow (4),$ respectively for $p=1$. 
Finally, in subsection \ref{sub-infty}, we prove the theorem for $p=\infty.$

\section{preliminaries}\label{sec-pre}
Throughout the article, the extreme contractions of $\mathcal{L}(\ell_1^n,\ell_1^m)$ play a crucial role. Recall that the extreme points of the closed unit ball of $\mathcal{L}(\ell_1^n,\ell_1^m)$ are called extreme contractions of the space. Suppose $E_{\mathcal{X}}$ denotes the set of all extreme points of the closed unit ball of a Banach $\mathcal{X}.$ It is easy to observe the following lemma.  For a proof in a more general case, see \cite[Lem. 2.8]{Sh}.
\begin{lemma}
	An operator $S\in \mathcal{L}(\ell_1^n,\ell_1^m)$ is an extreme contraction if and only if $S(E_{\ell_1^n})\subset E_{\ell_1^m}.$   
\end{lemma}
We use the following characterization of parallel pairs of $\mathcal{L}(\ell_1^n,\ell_1^m)$ without mentioning further.  A detailed study on parallel pairs on Banach spaces, and on operators spaces can be found in \cite{BCMWZ,MSP,NT,Se,ZM}. For $A\in \mathcal{L}(\ell_1^n,\ell_1^m),$ suppose $M_A$ denotes the collection of all unit vectors of $\ell_1^n,$ where $A$ attains its norm, that is, 
$$M_A:=\{x\in \ell_1^n:\|x\|=1,\|Ax\|_1=\|A\|\}.$$
\begin{theorem}\cite[Th. 2.3, Rem. 2.4]{MSP}
	Suppose $A,B\in \mathcal{L}(\ell_1^n,\ell_1^m).$ Then the following are equivalent.
	\begin{itemize}
		\item $A\parallel B$ (resp. $A,B$ form a TEA pair).
		\item There exists $x\in M_A\cap M_B$ such that $Ax\parallel Bx$ (resp. $Ax,Bx$ form a TEA pair).	
	\end{itemize}
\end{theorem}
The following lemma will also be used frequently. The lemma is proved for particular cases in \cite{LTWW,LTWW2}. Here we provide a proof for arbitrary separable Banach  spaces. Recall that a non-zero vector $x\in \mathcal{X}$ is said to be smooth if the set 
\[J(x):=\{f\in  \mathcal{X}^*:\|f\|=1,f(x)=\|x\| \}\]
is singleton, where $\mathcal{X}^*$ is the dual space of the Banach space $\mathcal{X}.$

\begin{lemma}\label{lem-pdim}
	Suppose $x\parallel y$ for all $x,y\in \mathcal{X},$ where $\mathcal{X}$ is a separable Banach space. Then $\dim(\mathcal{X})\leq 1.$	
\end{lemma}
\begin{proof}
	Since $\mathcal{X}$ is separable, the set of smooth vectors of $\mathcal{X}$ is non-empty (see \cite[P. 171]{H}). Let $x\in \mathcal{X}$ be smooth. Suppose $J(x)=\{f\}.$ Let $y\in \ker(f).$ Since $x\parallel y,$ assume $\|x+\lambda y\|=\|x\|+\|y\|,$ for some $\lambda\in \mathbb{T}.$ Now from \cite{NT} it follows that $f(\lambda y)=\|y\|,$ and so $y=0.$ Thus $\ker(f)=\{0\},$ consequently, $\dim(\mathcal{X})\leq 1.$  
\end{proof}

The following description of smooth vectors of $\mathcal{L}(\ell_1^n,\ell_1^m)$ is an integral part of the proof of  the implication $(3)\Rightarrow (4)$ of Theorem \ref{th-01}. For a detailed study on smooth operators see \cite[Ch. 6]{MPS}.

\begin{lemma}\label{lem-06}
	Suppose $A\in \mathcal{L}(\ell_1^n,\ell_1^m)$ is smooth and $\|A\|=1.$  Then there exists $1\leq j\leq n$ such that $\|Ae_j\|_1=1$ and $\|Ae_i\|_1<1$ for all $1\leq i~(\neq j)\leq n.$ Moreover, $Ae_j=(c_{1j},c_{2j},\ldots, c_{mj}),$ where $c_{kj}\neq 0$ for all $1\leq k\leq m.$
\end{lemma}
\begin{proof}
	Since $A$ is smooth, by \cite[Ch. 6]{MPS}, $M_{A}=\{\alpha x:\alpha\in \mathbb{T}\}$ for some unit vector  $x$ of $\ell_1^n.$ It is easy to observe that $x\in E_{\ell_1^n},$ for otherwise, if $x=\frac{1}{2}(x_1+x_2)$ for some unit vectors  $x_1,x_2$ of $ \ell_1^n,$ then $\|Ax_1\|_1=\|Ax_2\|_1=\|A\|=1,$ that is $x_1,x_2\in M_{A}.$ Thus $x_1=\alpha_1x, x_2=\alpha_2 x,$ for some $\alpha_1,\alpha_2\in \mathbb{T},$ and so $x=\frac{1}{2}(\alpha_1+\alpha_2)x.$ This clearly implies that $\alpha_1=\alpha_2=1.$  So there exists  $1\leq j\leq n$ such that $M_{A}=\{\alpha e_j:\alpha\in \mathbb{T}\}.$ Therefore, $\|Ae_j\|_1=1$ and $\|Ae_i\|_1<1$ for all $1\leq i~(\neq j)\leq n.$ Moreover, from \cite[Ch. 6]{MPS}, it follows that $Ae_j$ is smooth in $\ell_1^m.$ Note that,  if  $f=(b_1,b_2,\ldots,b_m)\in \ell_\infty^m$ is such that 
	\[b_k=\begin{cases}
	e^{-i\theta_{kj}}, &\text{ if } c_{kj}\neq0, \text{ and }c_{kj}=e^{i\theta_{kj}}|c_{kj}|\\
		\pm 1 ,&\text{ if } c_{kj}=0
	\end{cases},\]
	then $f\in J(Ae_j).$ Since $Ae_j$ is smooth, $J(Ae_j)$ is singleton, and so $c_{kj}\neq 0$ for all $1\leq k\leq m.$ 
\end{proof}

We use the symbols $e_i, \mathbf{e_i}$ to represent the $i$-th standard co-ordinate vectors of  $\mathbb{F}^n$ and $\mathbb{F}^m,$ respectively. 
For $1\leq i\leq m, 1\leq j\leq n,$ suppose $E_{ij}\in \mathcal{L}(\ell_1^n,\ell_1^m)$ is defined as 
\[E_{ij}e_k=\begin{cases}
	\mathbf{e_i},&\text{ if } k=j\\
	0,&\text{ if } 1\leq k\neq j\leq n
\end{cases}.\]

\section{Proof of Theorem \ref{th-01}}\label{sec-proof}
Throughout the subsections \ref{sub-12}, \ref{sub-23}, and \ref{sub-34}, we assume that $p=1.$
\subsection{Proof of  $(1)\Rightarrow(2)$}\label{sub-12}
We prove this implication using  the following lemma.
\begin{lemma}\label{lem-tea1}
	Suppose $T:\mathcal{L}(\ell_1^n,\ell_1^m)\to \mathcal{L}(\ell_1^n,\ell_1^m)$ is a linear map and $T$ preserves TEA pairs. Suppose there exists $A\in \ker(T)$ such that $M_{A}=\{\alpha e_j:\alpha\in \mathbb{T}\}$ for some $1\leq j\leq n.$ Then $T=0.$
\end{lemma}
\begin{proof}
	Let $\mathcal{B}=\{B\in \mathcal{L}(\ell_1^n,\ell_1^m):Be_j=0\}.$ First observe that $T(B)=0$ for all $B\in \mathcal{B}.$ Indeed, assume $\xi=\max\{\|Ae_i\|_1:1\leq i\neq j\leq n\}.$ Then $\xi<\|A\|.$ Now, $\eta A+B$ and $\eta A-B$ form a TEA pair for $\eta >\frac{\|B\|}{\|A\|-\xi}.$ Therefore, $T(\eta A+B)$ and $T(\eta A-B)$ form a TEA pair and so 
	\[0=\|T(\eta A+B)+T(\eta A-B)\|=\|T(B)\|+\|T(-B)\|=2\|T(B)\|,\]
	that is, $T(B)=0.$ In particular, $T(E_{ik})=0$ for all $1\leq i\leq m$ and for all $1\leq k\neq j\leq n.$ Now, repeating the above argument with $E_{ik},$ we get $T(E_{il})=0$ for all $1\leq i\leq m$ and for all $1\leq l\neq k\leq n.$ Therefore, $T(X)=0$ for all $X\in \mathcal{L}(\ell_1^n,\ell_1^m).$
\end{proof}

Now we are ready to prove $(1)\Rightarrow(2)$. 
\begin{proof}
	Since $T$ preserves TEA pairs, it clearly preserves parallel pairs. We only prove that $rank(T)>1.$ 	
If possible, assume that $rank(T)=1.$ 	
Let $S\in \mathcal{L}(\ell_1^n,\ell_1^m)$ be an arbitrary extreme contraction. We claim that $T(S)=0.$ If possible, assume $T(S)\neq 0.$ Choose $e\in E_{\ell_1^m}\setminus\{\mu Se_1:\mu\in \mathbb{T}\}.$ Define $U,V\in \mathcal{L}(\ell_1^n,\ell_1^m)$ as follows.
$$Ue_i=\begin{cases}e,& \text{if } i=1\\
Se_i,& \text{if }2\leq i\leq n 
\end{cases},
\quad Ve_i=\begin{cases}-e,& \text{if } i=1\\
	Se_i,& \text{if }2\leq i\leq n 
\end{cases} .$$
Then $U,V$ are extreme contractions of $ \mathcal{L}(\ell_1^n,\ell_1^m).$ Since $rank(T)=1,$ the sets $\{T(S),T(U)\}$ and $\{T(S),T(V)\}$ are linearly dependent. Let $$\alpha T(S)+\beta T(U)=0, \quad \text{ and }\quad  \gamma T(S)+\delta T(V)=0$$ for some $\alpha, \beta,\gamma,\delta \in \mathbb{F},$ where $(\alpha,\beta)\neq (0,0),$ and $(\gamma,\delta)\neq (0,0).$  If either $\beta=0,$ or $\delta =0,$ then clearly, $T(S)=0.$ So assume $\beta \neq 0, \delta\neq 0.$ Note that, if  $\alpha=\gamma=0,$ then $T(U)=T(V)=0,$ that is $T(U-V)=0.$ By Lemma \ref{lem-tea1}, $T=0,$ a contradiction. So either $\alpha\neq 0$ or $\gamma\neq 0.$ Without loss of generality, assume that  $\alpha\neq0.$ Note that, $$\|\alpha Se_1+\beta Ue_1\|_1=\|\alpha Se_1+\beta e\|_1=|\alpha|+|\beta|, \text{ and}$$ 
\[\|\alpha Se_i+\beta Ue_i\|_1=\|(\alpha +\beta )Se_i\|_1=|\alpha+\beta |, \text{ for } i>1.\] 
Thus, $\|\alpha S+\beta U\|=|\alpha|+|\beta|=\|\alpha S\|+\|\beta U\|.$
Consequently, $\alpha S$ and $\beta U$ form a TEA pair, and so do $\alpha T(S), \beta T(U).$ Therefore, $$0=\|\alpha T(S)+\beta T(U)\|=\|\alpha T(S)\|+\|\beta T(U)\|,$$ which proves that $T(S)=0.$ Since $T(S)=0$ for each extreme contraction $S\in \mathcal{L}(\ell_1^n,\ell_1^m),$ $T=0.$ This contradiction proves that $rank(T)>1.$
\end{proof}

\subsection{Proof of  $(2)\Rightarrow(3)$}\label{sub-23}
Let us first prove the following lemma, which will be used in this subsection.
\begin{lemma}\label{lem-inv}
Suppose $T:\mathcal{L}(\ell_1^n,\ell_1^m)\to \mathcal{L}(\ell_1^n,\ell_1^m)$ is a non-zero linear map and $T$ preserves parallel pairs. Let $A\in \ker(T)$ be such that $\|A\|=1,$ $e_i\in M_A,$ and $Ae_n=0.$  Then $T(X), T(E_{jn})$ are linearly dependent for any $1\leq j\leq m,$ where 
\[Xe_k=
\begin{cases}
	Ae_i,& \text{ if } k=i\\
	0,&\text{ if } 1\leq k\neq i\leq n
\end{cases}.\]
\end{lemma}
\begin{proof}
	 Let $1\leq j\leq m.$ Suppose $Y,Z\in span\{X,E_{jn}\},$ and $Y=\alpha X+\beta E_{jn},~ Z=\gamma X+\delta E_{jn}$ for some $\alpha,\beta,\gamma,\delta\in \mathbb{F}.$ We claim that $T(Y)\parallel T(Z).$ If $Y\parallel Z,$ then this is indeed the case. Suppose $Y$ is not parallel to $Z.$ 
	Then either $\frac{|\beta|}{|\alpha|}<1$ and $\frac{|\gamma|}{|\delta|}<1,$ or $\frac{|\alpha|}{|\beta|}<1$ and $\frac{|\delta|}{|\gamma|}<1.$ Without loss of generality, assume 
	 \begin{equation*}\label{eq-invlem}
	 	Y=X+\beta E_{jn}, \quad \text{and} \quad Z=\gamma X+E_{jn},~\text{ where } |\beta|<1,|\gamma|<1.
	 \end{equation*}   
 First suppose that $\gamma\neq 0.$ Choose $\xi>\frac{1}{|\gamma|}.$ Since $e_i\in M_A$ and $Ae_n=0,$  $i\neq n.$ Now,
 \[\|(\xi \gamma A+Z)e_i\|_1=\|(\xi\gamma+\gamma)Ae_i+E_{jn}e_i\|_1=|\gamma|(\xi+1),\] whereas
 $\|(\xi\gamma A+Z)e_n\|_1=\|E_{jn}e_n\|_1=1,$ and for $1\leq k\neq i\leq n-1,$
 \[\|(\xi\gamma A+Z)e_k\|_1=\|\xi\gamma Ae_k+E_{jn}e_k\|_1\leq |\gamma|\xi\|Ae_k\|_1\leq |\gamma|\xi.\]
 Since $|\gamma|(\xi+1)>|\gamma|\xi>1,$ we get $e_i\in M_{\xi\gamma A+Z}\cap M_Y.$ Moreover, $(\xi\gamma A+Z)e_i\parallel Ye_i.$ Therefore,
 \begin{eqnarray*}
 \xi\gamma A+Z \parallel Y\quad \Rightarrow \quad T(\xi \gamma A+Z)\parallel T(Y)\quad \Rightarrow\quad T(Z)\parallel T(Y).
 	\end{eqnarray*}
 Now, assume $\gamma=0.$ By the previous argument, we have 
 \begin{eqnarray*}
 &&	T(Y)\parallel T(\epsilon X+E_{jn}),\quad \text{ for all } 0<|\epsilon|<1\\
 	&\Rightarrow &T(Y)\parallel T(E_{jn}), \quad \text{letting } \epsilon \to 0\\
 	&\Rightarrow &T(Y)\parallel T(Z).
 	\end{eqnarray*}
 Thus in each case $T(Y)\parallel T(Z),$ that is, any two vectors of $span\{T(X),T(E_{jn})\}$ are parallel. Hence by Lemma \ref{lem-pdim}, $T(X),T(E_{jn})$ are linearly dependent.
\end{proof}

Next we prove the implication $(2)\Rightarrow (3).$

\begin{proof}
If possible, assume $T$ is not invertible. Then there exists $A\neq 0$ such that $T(A)=0.$ 
Let $k$ be the smallest positive integer such that $T(A)=0,$ $\|A\|=1$ and  $M_A\cap E_{\ell_1^n}$ contains exactly $k$ linearly independent vectors. Without loss of generality, assume that  $M_A\cap E_{\ell_1^n}=\{\xi e_i:1\leq i\leq k,\xi \in \mathbb{T}\}.$ \\

\textbf{Step 1:} We prove that $k\leq 2.$ \\
If $n=2,$ then this is trivially true. So assume $n\geq 3.$ If possible, let $k\geq 3.$ Define $U,V\in \mathcal{L}(\ell_1^n,\ell_1^m)$ as follows:
$$Ue_i=\begin{cases}Ae_1,& \text{if }i=1 \\
	0,& \text{if } 2\leq i\leq n\end{cases}, \quad Ve_i=\begin{cases}Ae_2,& \text{if } i=2\\
	0,& \text{if } 1\leq i\neq 2\leq n\end{cases} .$$
Let $B,C\in span\{U,V\}.$ Then $B=\alpha U+\beta V,~ C=\gamma U+\delta V$ for some $\alpha,\beta,\gamma,\delta\in \mathbb{F}.$
 If $B\parallel C,$ then $T(B)\parallel T(C).$ Suppose $B$ is not parallel $ C.$ 
 Without loss of generality, assume 
 \begin{equation}\label{eq-inv}
 B=U+\beta V, \quad \text{and} \quad C=\gamma U+V,~\text{ where } |\beta|<1,|\gamma|<1.
 \end{equation} Consider $$D=B-\frac{1+\beta}{2}A,\quad  \text{and}\quad   E=C-\frac{1+\gamma}{2}A.$$
If $\Re(\beta)\leq0,$ then  $e_2\in M_D\cap M_C,$ and $De_2\parallel Ce_2.$ Therefore, $D\parallel C,$ which implies that $T(D)\parallel T(C),$ that is, $T(B)\parallel T(C).$ Similarly, if $\Re(\gamma)\leq0,$ then since $e_1\in M_B\cap M_{E},$ and $Be_1\parallel Ee_1,$
\[B\parallel E \quad \Rightarrow \quad T(B)\parallel T(E)\quad \Rightarrow \quad T(B)\parallel T(C).\]
If $\Re(\beta)>0,\Re(\gamma)>0,$ then $e_3\in M_D\cap M_E,$ and $De_3\parallel Ee_3.$ Therefore,
\[D\parallel E\quad \Rightarrow \quad T(D)\parallel T(E)\quad \Rightarrow \quad T(B)\parallel T(C).\] 
Thus in each case $T(B)\parallel T(C),$ that is, any two elements of $span\{T(U),T(V)\}$ are parallel, and so by Lemma \ref{lem-pdim}, $T(U)$ and $T(V)$ are linearly dependent. Let $T(\mu U+\nu V)=0,$ for some  scalars $\mu,\nu.$ Clearly $M_{\mu U+\nu V}\cap E_{\ell_1^n}\subseteq \{\xi e_1,\xi e_2:\xi\in \mathbb{T}\},$ which contradicts the minimality of $k.$ Therefore $k\leq 2.$\\

\textbf{Step 2:} We prove that $k=1.$\\
From Step 1, $k\leq 2.$ If possible, assume that $k=2,$ that is, $\|Ae_1\|_1=\|Ae_2\|_1=1.$ We consider the cases $n=2$ and $n>2$ separately. In each case, we establish a contradiction.\\

\textbf{Case 2a:} Suppose $n=2.$  Define $P\in \mathcal{L}(\ell_1^2,\ell_1^m)$ as follows:
$$Pe_1=Ae_1, Pe_2=0.$$  Suppose $R,N\in span\{P,E_{12}\}.$ We first show that $T(R)\parallel T(N).$ If $R\parallel N,$ then $T(R)\parallel T(N).$ Suppose $R$ is not parallel to $N.$ Then as in Step 1,  we may assume $R=P+\beta E_{12},$ and $N=\gamma P+E_{12},$ where $\beta,\gamma\in \mathbb{F},$ and   $|\beta|<1,|\gamma|<1.$ Now $(A-R)e_1=0,$ and $(A-R)e_2=Ae_2-\beta \mathbf{e_1}\neq 0,$ since $\|Ae_2\|_1>\|\beta \mathbf{e_1}\|_1 .$ Thus $e_2\in M_{A-R}\cap M_{N},$ and $(A-R)e_2\parallel Ne_2.$ Consequently,
\[A-R\parallel N\quad \Rightarrow \quad T(A-R)\parallel T(N)\quad \Rightarrow \quad T(R)\parallel T(N).\] 
Therefore, any two elements of $span\{T(P),T(E_{12})\}$ are parallel, and so by Lemma \ref{lem-pdim}, $T(P), T(E_{12})$ are linearly dependent. If either $T(P)=0,$ or $T(E_{12})=0,$ then minimality of $k$ is contradicted. Therefore, $T(P)=\xi T(E_{12})$ for some non-zero scalar $\xi .$ Similarly, we get $T(P)=\eta T(E_{22})$ for some non-zero scalar $\eta .$ Thus, $T(\xi E_{12}-\eta E_{22})=0.$ Now, $M_{\xi E_{12}-\eta E_{22}}=\{\alpha e_2:\alpha\in \mathbb{T}\}$ contradicts the minimality of $k.$ \\

\textbf{Case 2b:} Suppose $n>2.$ Define $B,C$ as in Step 1. We claim that $T(B)\parallel T(C).$ If $B\parallel C,$ then this is trivially true. Suppose $B$ is not parallel to C. Then we may assume \eqref{eq-inv}. Proceeding similarly as Step 1, we get $T(B)\parallel T(C),$ if either $\Re(\beta)\leq 0$ or $\Re(\gamma)\leq0.$ Now, let $\Re(\beta)>0,\Re(\gamma)>0.$ Since $k=2,$ $$1>\max\{\|Ae_i\|_1:3\leq i\leq n\}.$$ Choose $r\in \mathbb{N}$ such that $$(1-\frac{\Re(\beta)}{r})>\max\{\|Ae_i\|_1:3\leq i\leq n\}.$$ It is easy to observe that $e_2\in M_{A-\frac{1}{r}B}\cap M_C,$ and $(A-\frac{1}{r}B)e_2\parallel Ce_2.$ Therefore, 
\[(A-\frac{1}{r}B)\parallel C\quad \Rightarrow \quad T(A-\frac{1}{r}B)\parallel T(C)\quad \Rightarrow \quad  T(B)\parallel T(C).\]
Now, as in Step 1, we conclude that there exist non-zero scalars $\mu,\nu$ such that $T(\mu U+\nu V)=0.$ Moreover, we may assume $\|\mu U+\nu V\|=1.$ Let $Q=\mu U+\nu V.$ Then $Qe_j=0,$ for all $3\leq j\leq n,$ and $e_i\in M_{Q}$ for some $i\in \{1,2\}.$ If $e_1\in M_Q,$ then define
\[Xe_j=
\begin{cases}
Qe_1,& \text{ if } j=1\\
0,&\text{ if } 2\leq j\leq n
\end{cases}.\]
If $e_2\in M_Q,$ then define 
\[Xe_j=\begin{cases}Qe_2,& \text{ if } j=2\\
	0,&\text{ if } 1\leq j\neq 2\leq n
\end{cases}.\]
From Lemma \ref{lem-inv}, it follows that $T(X),T(E_{jn})$ are linearly dependent for all $1\leq j\leq m.$
If either $T(X)=0$ or $T(E_{jn})=0$ for some $1\leq j\leq m,$ then the minimality of $k$ is contradicted. Therefore, assume $T(X)=\mu T(E_{1n})=\nu T(E_{2n})$ for some non-zero scalars $\mu,\nu.$ This yields that $T(\mu E_{1n}-\nu E_{2n})=0,$ which again contradicts the minimality of $k.$\\

\textbf{Step 3:} Suppose $\mathcal{V}=\{B\in \mathcal{L}(\ell_1^n,\ell_1^m):Be_1=0\}.$ We show that $dim(T(\mathcal{V}))\leq1.$ Let $B,C\in \mathcal{V}.$ Since $k=1,$ $\|Ae_1\|_1>\|Ae_i\|_1$ for all $2\leq i\leq n.$ Let $$\xi>\max\bigg\{\frac{\|Be_i\|_1}{\|Ae_1\|_1-\|Ae_i\|_1},\frac{\|Ce_i\|_1}{\|Ae_1\|_1-\|Ae_i\|_1}:2\leq i\leq n\bigg\}.$$
Then $e_1\in M_{\xi A+B}\cap M_{\xi A+C}$ and $(\xi A+B)e_1\parallel (\xi A+C)e_1.$ Therefore,
\[(\xi A+B)\parallel (\xi A+C)\quad \Rightarrow \quad T(\xi A+B)\parallel T(\xi A+C)\quad \Rightarrow \quad T(B)\parallel T(C).\]
Since any two elements of $T(\mathcal{V})$ are parallel, using  Lemma \ref{lem-pdim} we get $\dim(T(\mathcal{V}))\leq 1.$ \\

Suppose $\dim(T(\mathcal{V}))=1.$ Let  $T(\mathcal{V})=span\{T(E_{ij})\}$ for some $2\leq j\leq n,1\leq i\leq m,$ where $T(E_{ij})\neq0.$ Choose $E_{1r},E_{2r}\in \mathcal{V},$ such that $r\neq j.$ Then there exist scalars $\mu,\nu$ such that $T(\mu E_{1r}-\nu E_{2r})=0,$ and $(\mu,\nu)\neq (0,0).$ Now, repeating the argument of this step with $\mu E_{1r}-\nu E_{2r}$ in place of $A,$ we get $\dim(T(\mathcal{W}))\leq 1,$ where 
$$\mathcal{W}=\{B\in \mathcal{L}(\ell_1^n,\ell_1^m):Be_r=0\}.$$ 
Since for each $1\leq s\leq m,$ $E_{s1},E_{ij}\in \mathcal{W},$ we have $T(E_{s1})\in span\{T(E_{ij})\}.$ Thus, $T(C)\in span \{T(E_{ij})\}$ for all $C\in \mathcal{L}(\ell_1^n,\ell_1^m),$ which contradicts that $rank(T)>1.$\\

Now, suppose $\dim(T(\mathcal{V}))<1.$ Then $T(E_{ij})=0$ for all $2\leq j\leq n,1\leq i\leq m.$ Now repeating the argument of this step with $E_{12}$ in place of $A,$ we get $\dim(T(\mathcal{W}_1))\leq 1,$ where $$\mathcal{W}_1=\{B\in \mathcal{L}(\ell_1^n,\ell_1^m):Be_2=0\}.$$ 
Since, $T(C)\in T(\mathcal{W}_1)$ for all $C\in \mathcal{L}(\ell_1^n,\ell_1^m),$ we again get a contradiction that $rank(T)>1.$\\

\noindent Thus if $T$ preserves parallel pairs and $rank(T)>1,$ then $T$ must be invertible.

\end{proof}

\subsection{Proof of  $(3)\Rightarrow (4)$}\label{sub-34} To prove this implication, we need several technical  lemmas. 
Throughout this subsection,  we assume that $T:\mathcal{L}(\ell_1^n,\ell_1^m)\to \mathcal{L}(\ell_1^n,\ell_1^m)$ is an invertible linear map and $T$ preserves parallel pairs. 
Observe that if $T$ attains norm at each smooth unit vector of $\mathcal{L}(\ell_1^n,\ell_1^m),$ then since the smooth vectors are dense in $\mathcal{L}(\ell_1^n,\ell_1^m),$ it follows that $T$ attains norm at each unit vector, and so $T$ is a scalar multiple of an isometry. Therefore, our main goal here is to show that $T$ attains its norm at each smooth unit vector of $\mathcal{L}(\ell_1^n,\ell_1^m).$ To accomplish the goal, we first show that $T$ attains its norm at each extreme contraction of $\mathcal{L}(\ell_1^n,\ell_1^m).$
\begin{lemma}\label{lem-01}
	For each extreme contraction $S$ of $\mathcal{L}(\ell_1^n,\ell_1^m),$ $\frac{1}{\|T(S)\|}T(S)$ is an extreme contraction.
\end{lemma}
\begin{proof}
Suppose $S$ is an extreme contraction of $\mathcal{L}(\ell_1^n,\ell_1^m).$ Without loss of generality, assume that $\|T(S)\|=1.$ 
 We first prove  that $\|T(S)e_i\|_1=1$ for all $1\leq i\leq n.$	Let $A\in \mathcal{L}(\ell_1^n,\ell_1^m)$ and $e_j\in M_A$ for some $1\leq j\leq n.$ Since $S$ is an extreme contraction, $Se_j\in E_{\ell_1^m}.$ Moreover, each vector of $E_{\ell_1^m}$ is parallel to all vectors of $\ell_1^m.$ Therefore $Se_j\| Ae_j,$ and so $S\parallel A,$ which implies that $T(S)\parallel T(A)$ for all $A\in \mathcal{L}(\ell_1^n,\ell_1^m).$ Since $T$ is invertible, there exists $A\in \mathcal{L}(\ell_1^n,\ell_1^m)$ such that $T(A)=E_{1i}.$  Clearly $M_{E_{1i}}=\{\alpha e_i:\alpha\in \mathbb{T}\}.$ Now, $T(S)\parallel E_{1i}$ implies that $M_{T(S)}\cap M_{E_{1i}}\neq \emptyset,$ and so $$1=\|T(S)\|=\|T(S)e_i\|_1.$$
Next we show that $T(S)e_i\in E_{\ell_1^m}$ for all $i.$ If possible, assume $T(S)e_i\notin E_{\ell_1^m}.$ Without loss of generality, assume $c_{1i}\neq 0$ and $c_{2i}\neq 0,$ where $T(S)e_i=(c_{1i},\ldots,c_{mi}).$ Consider $A\in \mathcal{L}(\ell_1^n,\ell_1^m),$ where $Ae_i=(-c_{1i},c_{2i}\ldots,c_{mi})$ and $Ae_k=0$ for all $1\leq k\neq i\leq n.$ Clearly $M_A=\{\alpha e_i:\alpha\in \mathbb{T}\}.$ Now, 
\[T(S)\parallel A \quad \Rightarrow \quad T(S)e_i\parallel Ae_i.\] Let $\lambda\in \mathbb{T} $ be  such that  $\|T(S)e_i+\lambda Ae_i\|_1=\|T(S)e_i\|_1+\|Ae_i\|_1.$ Then 
\[|1-\lambda||c_{1i}|+\sum_{j=2}^m|1+\lambda||c_{ji}|=2,\] which implies that $|1-\lambda|=|1+\lambda|=2.$ This contradiction proves that $T(S)e_i\in E_{\ell_1^m}.$ Therefore, $T(S)$ is an extreme contraction. 
\end{proof}

\begin{lemma}\label{lem-07}
	Suppose $A,B$ are extreme contractions of $\mathcal{L}(\ell_1^n,\ell_1^m)$ satisfying the following.
	\begin{enumerate}[\rm(i)]
		\item $Ae_i=Be_i$ for all $i\in J,$ where $J\subset \{1,\ldots,n\},$ and $J$ contains $n-1$ elements.
		\item $\{Ae_i,Be_i\}$ is linearly independent if $i\notin J.$
		\item $\frac{1}{\|T(A)\|}T(A)e_j=\mu_j\frac{1}{\|T(B)\|}T(B)e_j$  for all $1\leq j\leq n,$ where $\mu_j\in \mathbb{T}.$
	\end{enumerate}
Then $\Re(\mu_k)=\Re(\mu_j)$ for all $1\leq j, k\leq n.$
\end{lemma}
\begin{proof}
	Without loss of generality, assume that $J=\{1,2,\ldots,n-1\}.$ 
	Let $\|T(A)\|=k_1,\|T(B)\|=k_2.$ Using Lemma \ref{lem-01}, note that $$\|T(A+B)\|=\max\{|k_1+\overline{\mu_i}k_2|:1\leq i\leq n\}.$$  If possible, assume that $\Re(\mu_k)\neq \Re(\mu_j)$ for some $1\leq j,k\leq n.$ Then there exists $i$ such that $e_i\notin M_{T(A+B)}.$ Without loss of generality, assume that $e_n\notin M_{T(A+B)}.$ Let $\mathcal{V}=span\{E_{in}:1\leq i\leq n\}.$ Then for all $C\in \mathcal{V},$ $M_C=\{\lambda e_n:\lambda\in \mathbb{T}\}.$ Therefore if $C\in \mathcal{V},$ then $T(A+B)$ is not parallel to $C,$  and so $A+B$ is not parallel to $T^{-1}(C).$ Note that, for all $1\leq i\leq n-1,$ $e_i\in M_{A+B}$ and $$(A+B)e_i\parallel T^{-1}(C)e_i,$$ since $\frac{1}{2}(Ae_i+Be_i)=Ae_i\in E_{\ell_1^m}.$ So if $e_i\in M_{T^{-1}(C)},$ then $A+B\parallel T^{-1}(C),$ a contradiction. Therefore, for all $C\in \mathcal{V},$ $M_{T^{-1}(C)}=\{\lambda e_n:\lambda\in \mathbb{T}\}.$ Suppose for $1\leq i\leq m,$
	\[T^{-1}(E_{in})e_n=\alpha_{i1}Ae_n+\alpha_{i2}Be_n+u_i,\]
	where $\alpha_{i1},\alpha_{i2}\in \mathbb{F},$ and $ u_i\in \ell_1^m$ such that $\|\beta Ae_n+\gamma Be_n+u_i\|_1=|\beta|+|\gamma|+\|u_i\|_1$ for all $\beta,\gamma \in \mathbb{F}.$ It is easy to observe that if $\alpha_{i1}=0,$ then 
	$$Ae_n+Be_n\parallel \alpha_{i1}Ae_n+\alpha_{i2}Be_n+u_i.$$ 
	Indeed, let $\alpha_{i1}=0,$ and $\alpha_{i2}=e^{i\theta}|\alpha_{i2}|.$ Then 
	\begin{eqnarray*}
		&&\|(Ae_n+Be_n)+e^{-i\theta}(\alpha_{i2}Be_n+u_i)\|_1\\
		&=&2+|\alpha_{i2}|+\|u_i\|_1\\
		&=&\|Ae_n+Be_n\|_1+\|\alpha_{i2}Be_n+u_i\|_1,
	\end{eqnarray*}
	where the last equality holds since $Ae_n,$ and $Be_n$ are linearly independent extreme points of the closed unit ball of $\ell_1^m.$ Thus, if $\alpha_{i1}=0,$ then 
	\begin{eqnarray*}
		&&	Ae_n+Be_n\parallel \alpha_{i1}Ae_n+\alpha_{i2}Be_n+u_i\\
		&\Rightarrow  &(A+B)e_n \parallel T^{-1}(E_{in})e_n\\
		&\Rightarrow  & (A+B) \parallel T^{-1}(E_{in}),
	\end{eqnarray*}
	a contradiction. Assume $\alpha_{11}\neq 0,\alpha_{21}\neq 0.$ Choose $\beta=-\frac{\alpha_{11}}{\alpha_{21}}.$ Then 
	\[T^{-1}(E_{1n}+\beta E_{2n})e_n=(\alpha_{12}+\beta \alpha_{22})Be_n+(u_1+\beta u_2),\]
	and so \[(A+B)\parallel T^{-1}(E_{1n}+\beta E_{2n}),\] a contradiction, since $E_{1n}+\beta E_{2n}\in \mathcal{V}.$ Thus we must have $$\Re(\mu_j)=\Re(\mu_k) \text{ for all }1\leq j,k\leq n.$$ 
\end{proof}

The following lemma shows that the image of certain extreme contractions of $\mathcal{L}(\ell_1^n,\ell_1^m)$ under $T$ satisfy some special properties.
From the next lemma it can be easily proved that $T$ attains norm at each extreme contraction. Although the following proofs hold for $\mathbb{F}\in \{\mathbb{R},\mathbb{C}\}$, we would like to mention that significantly simple alternative proofs can be provided if $\mathbb{F}=\mathbb{R}.$ In what follows $M_T$ denotes the set of all unit vectors of $\mathcal{L}(\ell_1^n,\ell_1^m)$ where $T$ attains its norm, that is,
\[M_T:=\{A\in \mathcal{L}(\ell_1^n,\ell_1^m) :\|A\|=1,\|T(A)\|=\|T\|\}.\]

\begin{lemma}\label{lem-md}
	Let $S_1,S_2$ be extreme contractions of $\mathcal{L}(\ell_1^n,\ell_1^m)$ such that $S_1\in M_T.$ 
	Suppose there exist $1\leq l,r\leq n,$ where $r\neq l$ satisfying the following.
	
	\rm(i) Either $\{S_1e_l,S_2e_l\}$ is linearly independent or $S_1e_l=S_2e_l.$
	
	\rm(ii) $\{S_1e_r,S_2e_r\}$ is linearly independent.\\
	 Then $S_2\in M_T.$ Moreover, at least one of the following holds.
	\begin{enumerate}[(I)]
		\item There exists $1\leq j\leq n$ such that  $\{T(S_1)e_j,T(S_2)e_j\}$ is linearly independent.
		\item   There exists $1\leq i,j\leq n$ such that $T(S_1)e_i=T(S_2)e_i,$ and  $T(S_1)e_j=-T(S_2)e_j.$ 
	\end{enumerate}
\end{lemma}
\begin{proof}
	Without loss of generality, assume that $l=1,$ and $r=n.$ 
Define
	\[S_3e_i=\begin{cases}
		S_1e_i ,&\text{ if } 1\leq i\leq n-1\\
		-S_2e_n,& \text{ if } i=n 
	\end{cases}, \quad
	S_4e_i=\begin{cases}
		S_2e_i, &\text{ if } 1\leq i\leq n-1\\
		-S_1e_n,& \text{ if } i=n
	\end{cases} .\]
Note that in this case, $S_i ~(1\leq i\leq 4)$ satisfies the following properties. 
	\begin{enumerate}[(P1)]
		\item 	$S_1+S_4=S_2+S_3.$
		\item $S_3,S_4$ are extreme contractions of $\mathcal{L}(\ell_1^n,\ell_1^m),$ and $S_2+S_3\neq 0.$
		\item $S_1-S_4\parallel S_1+S_2,$ and so  $T(S_1-S_4)\parallel T(S_1+S_2).$ \\
		Indeed, since $e_n\in M_{S_1-S_4}\cap M_{S_1+S_2},$ and $(S_1-S_4)e_n\parallel (S_1+S_2)e_n,$ we get  $S_1-S_4\parallel S_1+S_2.$ 
		\item $S_1-S_4\parallel S_1+S_3,$ and so  $T(S_1-S_4)\parallel T(S_1+S_3).$\\
		Since $e_n\in M_{S_1-S_4}\cap M_{S_1+S_3},$ and $(S_1-S_4)e_n\parallel (S_1+S_3)e_n,$ we have $S_1-S_4\parallel S_1+S_3.$ 
		\item $S_1+S_2\parallel S_1+S_3,$ and so  $T(S_1+S_2)\parallel T(S_1+S_3).$\\
		Since $e_1\in M_{S_1+S_2}\cap M_{S_1+S_3},$ and $(S_1+S_2)e_1\parallel (S_1+S_3)e_1,$ we have $S_1+S_2\parallel S_1+S_3.$ 
	\end{enumerate}
Let $k_j=\|T(S_j)\|$ for $1\leq j\leq 4.$ Since $T$ is invertible, $k_j\neq 0$ for all $1\leq j\leq 4.$ 
 By Lemma \ref{lem-01}, for all $1\leq r\leq n,$ suppose
\[\frac{T(S_1)}{k_1}e_r=x_r, \quad \frac{T(S_2)}{k_2}e_r=y_r,\quad \frac{T(S_3)}{k_3}e_r=z_r,\quad \frac{T(S_4)}{k_4}e_r=w_r,\]
where $x_r,y_r,z_r,w_r\in E_{\ell_1^m}.$  Using (P1), we get
\begin{equation}\label{eq-c01}
	x_rk_1+w_rk_4=y_rk_2+z_rk_3,\quad \forall\quad 1\leq r\leq n. 
\end{equation}
\textbf{Case a:} Suppose (I) holds. In this case, it only remains to show that $S_2\in M_T,$ that is, $k_1=k_2.$ Now, there are two subcases. We deal with these subcases separately.

Subcase a.1: There exists $1\leq r\leq n$ such that $\{x_r,w_r\}$ is linearly independent.

Subcase a.2: $\{x_r,w_r\}$ is linearly dependent  for all $1\leq r\leq n.$\\

\noindent \textbf{Subcase a.1:} Suppose $\{x_r,w_r\}$ is linearly independent for some $r$.  If $\{y_r,z_r\}$ is linearly dependent, then from \eqref{eq-c01}, it follows that either $k_1=0$ or $k_4=0,$ a contradiction. So   $\{y_r,z_r\}$ is linearly independent. Now by \eqref{eq-c01}, we get either
\begin{equation}\label{eq-c115}
x_r=y_r,\quad z_r=w_r, \quad k_1=k_2, \quad  k_3=k_4
\end{equation} or 
\begin{equation}\label{eq-c02}
	x_r=z_r, \quad y_r=w_r,\quad  k_1=k_3,\quad  k_2=k_4.
\end{equation}
If \eqref{eq-c115} holds, then we are done.  So suppose \eqref{eq-c02} holds. By (P3), there exists 
 $e_i\in M_{T(S_1-S_4)}\cap M_{T(S_1+S_2)}$ such that $T(S_1-S_4)e_i\parallel T(S_1+S_2)e_i.$  Suppose $\lambda\in \mathbb{T},$ and 
\begin{eqnarray*}
	&&\|T(S_1-S_4)\|+\| T(S_1+S_2)\|\\
	&=&\|T(S_1-S_4)e_i\|_1+\| T(S_1+S_2)e_i\|_1\\
	&=&\|T(S_1-S_4)e_i+\lambda T(S_1+S_2)e_i\|_1\\
	&=&\|(1+\lambda)x_ik_1-w_ik_4+\lambda y_ik_2\|_1.
\end{eqnarray*}
Note that 
\[k_1+k_4\geq\|T(S_1-S_4)\|\geq \|T(S_1-S_4)e_r\|_1=\|x_rk_1-w_rk_4\|_1=k_1+k_4,\]
and 
\[k_1+k_2\geq\|T(S_1+S_2)\|\geq \|T(S_1+S_2)e_r\|_1=\|x_rk_1+y_rk_2\|_1=k_1+k_2.\]
Therefore, 
\begin{eqnarray*}
	&&2k_1+k_2+k_4\\
	&=&\|(1+\lambda)x_ik_1-w_ik_4+\lambda y_ik_2\|_1\\
	&\leq &|1+\lambda| k_1+k_2+k_4\\
	&\leq& 2k_1+k_2+k_4,
\end{eqnarray*}
which implies that $\lambda=1,$ and 
\begin{eqnarray}\label{eq-c112}
	\begin{split}
		&&2k_1+k_2+k_4\\
		&=&\|2x_ik_1-w_ik_4+ y_ik_2\|_1\\
		&\leq &\|2x_ik_1+y_ik_2\|_1+k_4\\
		&\leq& 2k_1+k_2+k_4.
	\end{split}
\end{eqnarray}
So $\|2x_ik_1+y_ik_2\|_1=2k_1+k_2,$ that is, either $\{x_i,y_i\}$ is linearly independent or $x_i=y_i.$\\

 If $\{x_i,y_i\}$ is independent, then from \eqref{eq-c01}, it follows that  
either 
\begin{equation}\label{eq-c116}
x_i=-w_i, \quad y_i=-z_i, \quad k_1=k_4, \quad k_2=k_3
\end{equation}
 or 
\begin{equation}\label{eq-c111}
	x_i=z_i, \quad y_i=w_i, \quad k_1=k_3, \quad k_2=k_4.
\end{equation}
If possible, assume  \eqref{eq-c111} holds. Using \eqref{eq-c111} in \eqref{eq-c112}, we have 
\[2k_1+k_2+k_4=\|2x_ik_1-w_ik_4+y_ik_2\|_1=2k_1,\]
that is $k_2=k_4=0,$ a contradiction.
Therefore \eqref{eq-c116} holds.  Using $k_1=k_4$ in \eqref{eq-c02}, we get $k_1=k_2.$ \\

 Suppose $x_i=y_i.$  Proceeding similarly, from \eqref{eq-c112} we have $\|2x_ik_1-w_ik_4\|_1=2k_1+k_4,$ that is, either $\{x_i,w_i\}$ is linearly independent or $x_i=-w_i.$ Now, $\{x_i,w_i\}$ is independent implies that either \eqref{eq-c115} or \eqref{eq-c02} holds with the index replaced by $i$.  Again if \eqref{eq-c115} holds, then we are done. Otherwise, from \eqref{eq-c112} we  get a contradiction. On the other hand, suppose $x_i=-w_i.$ Clearly $\{y_i,z_i\}$ is dependent. Suppose $y_i=\mu z_i,$ for some $\mu\in \mathbb{T}.$ By \eqref{eq-c01}, 
\begin{eqnarray*}
	w_i(- k_1+k_4)&=&z_i(\mu k_2+ k_3)\\
	\Rightarrow \|w_i(- k_1+k_4)\|_1&=&\|z_i(\mu k_2+ k_3)\|_1\\
	\Rightarrow |-k_1+k_4|&=&|\mu k_4+k_1|, \quad (\text{by } \eqref{eq-c02})\\
	\Rightarrow\mu&= & -1.
\end{eqnarray*}
Thus,
\[x_i=y_i=-z_i=-w_i.\]
Using the above equality in \eqref{eq-c01} and \eqref{eq-c02}, we get $k_1=k_2.$\\

\noindent \textbf{Subcase a.2:} Suppose  $\{x_r,w_r\}$ is linearly dependent  for all $1\leq r\leq n.$ Since (I) holds, that is, $\{x_j,y_j\}$ is linearly independent, so we must have
\begin{equation}\label{eq-c117}
	x_j=-w_j, \quad y_j=-z_j, \quad k_1=k_4, \quad k_2=k_3
\end{equation}
By (P5), there exists $e_r\in M_{T(S_1+S_2)}\cap M_{T(S_1+S_3)}$ such that $$T(S_1+S_2)e_r\parallel T(S_1+S_3)e_r.$$
Suppose $\lambda\in \mathbb{T}$ be such that 
\begin{eqnarray*}
	&&\|T(S_1+S_2)\|+\| T(S_1+S_3)\|\\
	&=&\|T(S_1+S_2)e_r\|_1+\| T(S_1+S_3)e_r\|_1\\
	&=&\|T(S_1+S_2)e_r+\lambda T(S_1+S_3)e_r\|_1\\
	&=&\|(1+\lambda)x_rk_1+y_rk_2+\lambda z_rk_3\|_1.
\end{eqnarray*}
Since \[k_1+k_2\geq\|T(S_1+S_2)\|\geq \|T(S_1+S_2)e_j\|_1=\|x_jk_1+y_jk_2\|_1=k_1+k_2,\]
and 
\[k_1+k_3\geq\|T(S_1+S_3)\|\geq \|T(S_1+S_3)e_j\|_1=\|x_jk_1+z_jk_3\|_1=k_1+k_3,\]
we have
\begin{equation}\label{eq-c118}
2k_1+k_2+k_3=\|(1+\lambda)x_rk_1+y_rk_2+\lambda z_rk_3\|_1.
\end{equation}
Now as previous, it is easy to observe that $$\lambda=1,\quad \|2x_rk_1+y_rk_2\|_1=2k_1+k_2, \quad \|2x_rk_1+z_rk_3\|_1=2k_1+k_3.$$ 
Thus, either $\{x_r,y_r\}$ is linearly independent or $x_r=y_r,$ and either $\{x_r,z_r\}$ is linearly independent or $x_r=z_r.$ If either $\{x_r,y_r\}$ or $\{x_r,z_r\}$  is linearly independent, then since $\{x_r,w_r\}$ is dependent, we must have \eqref{eq-c117} for the index $r,$ which contradicts \eqref{eq-c118}. Therefore, $x_r=y_r=z_r.$ Assume that $w_r=\mu x_r,$ for some $\mu\in \mathbb{T}.$ Then from \eqref{eq-c01} and \eqref{eq-c117}, it follows that 
\begin{eqnarray*}
(1	+\mu )k_1=2k_2 \quad \Rightarrow\quad  \Im(\mu)=0 \quad \Rightarrow\quad k_1=k_2.
\end{eqnarray*}

\textbf{Case b:} Suppose (I) does not hold, that is $\{x_i,y_i\}$ is linearly dependent for all $1\leq i\leq n.$ 

\textbf{Step 1.} In this step, we prove that there exists $1\leq j\leq n$ such that $\{x_j,w_j\}$ is linearly independent. If possible, assume that $\{x_i,w_i\}$ is linearly dependent for all $i.$ Then  from \eqref{eq-c01}, it is easy to observe that  $\{x_i,z_i\}$ is also linearly dependent for all $i.$ Suppose for $1\leq i\leq n,$ $\lambda_i,\mu_i,\nu_i\in \mathbb{T}$ be such that
\[x_i=\lambda_i y_i=\mu_iz_i=\nu_i w_i. \]
Thus from \eqref{eq-c01}, 
\begin{equation}\label{eq-c119}
	\begin{split}
&k_1=\overline{\lambda_i}k_2+\overline{\mu_i}k_3-\overline{\nu_i}k_4, ~\forall~1\leq i\leq n\\
\Rightarrow ~& \Re(\lambda_i-\lambda_j)k_2+\Re(\mu_i-\mu_j)k_3+\Re(\nu_j-\nu_i)k_4=0, \text{ and }\\
 & \Im(\lambda_i-\lambda_j)k_2+\Im(\mu_i-\mu_j)k_3+\Im(\nu_j-\nu_i)k_4=0, ~\forall~1\leq i,j\leq n.
\end{split}
\end{equation}
Since by Lemma \ref{lem-07}, $\Re(\mu_i)=\Re(\mu_j)$ for all $1\leq i,j\leq n,$  
\begin{equation}\label{eq-c20}
\Re(\lambda_i-\lambda_j)k_2+\Re(\nu_j-\nu_i)k_4=0.
\end{equation}
Now, using (P3) we get $1\leq r\leq n$ such that $e_r\in M_{T(S_1-S_4)}\cap M_{T(S_1+S_2)}.$ 
For all $1\leq j\leq n,$
\begin{eqnarray*}
\|T(S_1-S_4)e_r\|_1&\geq &	\|T(S_1-S_4)e_j\|_1\\
\Rightarrow |k_1-\overline{\nu_r}k_4|&\geq& |k_1-\overline{\nu_j}k_4|\\
\Rightarrow \Re(\nu_j)&\geq&\Re(\nu_r).
\end{eqnarray*}
Similarly, \[\|T(S_1+S_2)e_r\|_1\geq 	\|T(S_1+S_2)e_r\|_1
\quad\Rightarrow \quad \Re(\lambda_r)\geq\Re(\lambda_j).\]
Using these in \eqref{eq-c20}, we have $\Re(\lambda_j)=\Re(\lambda_r)$ and $\Re(\nu_j)=\Re(\nu_r)$ for all $1\leq j\leq n.$ Now from \eqref{eq-c119}, it is easy to observe that for each $j,$ either 
\[\lambda_j=\lambda_1, \quad \mu_j=\mu_1, \quad \nu_j=\nu_1\] or  
\[\lambda_j=\overline{\lambda_1}, \quad \mu_j=\overline{\mu_1}, \quad \nu_j=\overline{\nu_1}.\]
Choose real scalars $\xi,\eta,\zeta$ not all zero simultaneously such that 
\begin{eqnarray*}
\xi k_1+\eta \Re(\lambda_1)k_2+\zeta \Re(\mu_1)k_3&=&0, ~\text{ and }\\
	\eta \Im(\lambda_1)k_2+\zeta \Im(\mu_1)k_3&=&0.
\end{eqnarray*}
Therefore,
\begin{eqnarray*}
	\xi k_1+\eta \lambda_1 k_2+\zeta \mu_1k_3&=&0, ~\text{ and }\\
\xi k_1+	\eta \overline{\lambda_1}k_2+\zeta \overline{\mu_1}k_3&=&0,
\end{eqnarray*}
and so 
\begin{eqnarray*}
	\xi k_1+	\eta \overline{\lambda_j}k_2+\zeta \overline{\mu_j}k_3&=&0, ~\text{ for all } 1\leq j\leq n\\
	\Rightarrow 	(\xi k_1+	\eta \overline{\lambda_j}k_2+\zeta \overline{\mu_j}k_3)x_j&=&0, ~\text{ for all } 1\leq j\leq n\\
	\Rightarrow 	(\xi T(S_1)+	\eta T(S_2)+\zeta T(S_3))e_j&=&0, ~\text{ for all } 1\leq j\leq n\\
	\Rightarrow 	\xi T(S_1)+	\eta T(S_2)+\zeta T(S_3)&=&0\\
	\Rightarrow 	\xi S_1+	\eta S_2+\zeta S_3&=&0\\
		\Rightarrow 	(\xi S_1+	\eta S_2+\zeta S_3)e_n&=&0\\
			\Rightarrow 	\xi S_1e_n+	(\eta-\zeta) S_2e_n&=&0\\
			\Rightarrow \xi=0, \text{ and } \eta-\zeta&=&0, 
		\end{eqnarray*}
	 since  $\{S_1e_n,S_2e_n\}$  is linearly independent. Therefore $S_2+S_3=0,$ which contradicts (P2). This completes the proof of Step 1.
	 
	 \textbf{Step 2.} Using Step 1, suppose $\{x_j,w_j\}$ is linearly independent for some $j.$ Since $\{x_j,y_j\}$ is linearly dependent, by \eqref{eq-c01} we must have,
	 \begin{equation}\label{eq-c21}
	 x_j=y_j,\quad w_j=z_j, \quad k_1=k_2, \quad k_3=k_4.
	 \end{equation}
	 To prove (II), it only remains to show that $x_i=-y_i$ for some $i.$ Clearly,
	 \[\|T(S_1-S_4)\|\geq \|T(S_1-S_4)e_j\|_1=\|x_jk_1-w_jk_4\|_1=k_1+k_4,\] and so $\|T(S_1-S_4)\|=k_1+k_4.$ Similarly, $\|T(S_1+S_3)\|=k_1+k_3.$
By (P4), let $e_i\in M_{T(S_1-S_4)}\cap M_{T(S_1+S_3)}$ be such that 
\[T(S_1-S_4)e_i\parallel T(S_1+S_3)e_i.\] Suppose $\lambda\in \mathbb{T}$ such that 
\[\|T(S_1-S_4)e_i\|_1+\|T(S_1+S_3)e_i\|_1=\|T(S_1-S_4)e_i+\lambda T(S_1+S_3)e_i\|_1.\]
Therefore,
\begin{equation}\label{eq-c22}
	\begin{split}
2k_1+k_3+k_4&=\|T(S_1-S_4)\|+\|T(S_1+S_3)\|	\\
&=\|T(S_1-S_4)e_i\|_1+\|T(S_1+S_3)e_i\|_1\\
&=\|(1+\lambda)x_ik_1+\lambda z_ik_3-w_ik_4\|_1.
\end{split}
\end{equation}
Now it is easy to observe that 
\[\lambda=1,\quad \|2x_ik_1+z_ik_3\|_1=2k_1+k_3, \quad \|2x_ik_1-w_ik_4\|_1=2k_1+k_4.\]
Thus, either $\{x_i,z_i\}$ is linearly independent or $x_i=z_i,$ and either $\{x_i,w_i\}$ is linearly independent or $x_i=-w_i.$ Since $\{x_i,y_i\}$ is linearly dependent, so if either $\{x_i,z_i\}$  or $\{x_i,w_i\}$ is linearly independent, then from \eqref{eq-c01} it follows that \eqref{eq-c21} holds for the index $i,$ which contradicts \eqref{eq-c22}. Therefore, we must have $$x_i=z_i=-w_i.$$ Using this equality together with $k_1=k_2,$ and $k_3=k_4$ in \eqref{eq-c01}, it is straightforward to check that $x_i=-y_i.$
\end{proof}

\begin{lemma}\label{lem-04}
	Let $1\leq j\leq n,~1\leq i,k\leq m,$ and $i\neq k.$ Then\\
	\rm(i) $T$ attains norm at $E_{ij}.$\\
\rm(ii) $T$ attains norm at $\frac{1}{2}(\xi E_{ij}+\eta E_{kj}),$ for all $\xi ,\eta\in \mathbb{T}.$ \\
\rm(iii) if $B=\alpha E_{ij}+\beta E_{kj},$ then $\|T(B)\|=(|\alpha |+|\beta |)\|T\|,$ for all non-zero $\alpha,\beta \in \mathbb{F}.$
\end{lemma}
\begin{proof}
	 Suppose $S$ is an extreme contraction of $\mathcal{L}(\ell_1^n,\ell_1^m)$ such that $\|T(S)\|=\|T\|.$ 	\\
	 
	(i)  	If necessary multiplying $S$ by a unimodular scalar, we assume that either $\{Se_j,\mathbf{e_i}\}$ is linearly independent or $Se_j=\mathbf{e_i}.$ Define extreme contractions $S_1,S_2,S_3,$ $S_4\in \mathcal{L}(\ell_1^n,\ell_1^m)$ as follows:
	\[S_1e_r=\begin{cases}
		 \mathbf{e_i},&\text{ if } r=j\\
		Se_r,&\text{ if } 1\leq r\neq j\leq n
	\end{cases}, \quad S_2e_r=\begin{cases}
		\mathbf{e_i},&\text{ if } r=j\\
		-Se_r,&\text{ if } 1\leq r\neq j\leq n
	\end{cases},\]
	
	\[S_3e_r=\begin{cases}
		\mathbf{e_k}, &\text{ if } r=j\\
		Se_r,& \text{ if } 1\leq r\neq j\leq n
	\end{cases}, \quad
	S_4e_r=\begin{cases}
		\mathbf{e_k}, &\text{ if } r=j\\
		-Se_r,& \text{ if } 1\leq r\neq j\leq n
	\end{cases} .\]
Observe that $E_{ij}=\frac{1}{2}(S_1+S_2).$
	From Lemma \ref{lem-md} it clearly follows that $$\|T(S_1)\|=\|T(S_3)\|=\|T(S_4)\|=\|T\|.$$    Let $k=\|T(S_2)\|.$  Suppose 	for all $1\leq r\leq n,$
	\[\frac{T(S_1)}{\|T\|}e_r=x_r, \quad \frac{T(S_2)}{k}e_r=y_r,\quad \frac{T(S_3)}{\|T\|}e_r=z_r,\quad \frac{T(S_4)}{\|T\|}e_r=w_r,\]
	where $x_r,y_r,z_r,w_r\in E_{\ell_1^m}.$ Since $S_1+S_4=S_2+S_3,$
	\begin{equation}\label{eq-eij}
	x_r\|T\|+w_r\|T\|=y_rk+z_r\|T\|, \quad \forall~~~ 1\leq r\leq n.
	\end{equation}
	If possible, assume that $\{x_r,z_r\},$ and  $\{x_r,w_r\}$ are linearly dependent for all $1\leq r\leq n.$ Then $\{x_r,y_r\}$ is also linearly dependent. Suppose 
	\[x_r=\lambda_ry_r=\mu_rz_r=\nu_rw_r, \quad \forall~~1\leq r\leq n.\]
	Since $\{S_1e_j,S_3e_j\}$ is linearly independent, and $S_1e_r=S_3e_r$ for $r\neq j,$ from Lemma \ref{lem-md}, we have $x_t=z_t,$ and $x_l=-z_l$ for some $1\leq l,t\leq n,$ that is $\mu_t=1,$ and $\mu_l=-1.$   However, from Lemma \ref{lem-07} it follows that $\Re(\mu_r)$ is constant for all $r,$ which is a contradiction. Therefore, there exists $1\leq r\leq n,$ such that either 
	 $\{x_r,z_r\}$ or $\{x_r,w_r\}$   is linearly independent. Now from \eqref{eq-eij} it clearly follows that $k=\|T\|,$ and either $x_r=y_r$ or $\{x_r,y_r\}$ is linearly independent. Therefore,
	 \begin{eqnarray*}
	 \|T\|\geq\|T(E_{ij})\|=\frac{1}{2}\|T(S_1+S_2)\|&\geq&\frac{1}{2}\|T(S_1+S_2)e_r\|_1\\
	 &=&\frac{1}{2}\|x_r+y_r\|_1\|T\|=\|T\|.
	 \end{eqnarray*}

(ii)  Choose $1\leq l\leq n$ such that $l\neq j,$ and choose $e\in E_{\ell_1^m}$ such that $\{Se_l,e\}$ is linearly independent.	
Note that $$\frac{1}{2}(\xi E_{ij}+\eta E_{kj})=\frac{1}{2}(B_{ij}+C_{kj}),$$ where 
\[B_{ij}e_r=\begin{cases}
	\xi \mathbf{e_i},&\text{ if } r=j\\
	e,&\text{ if } 1\leq r\neq j\leq n
	\end{cases}, \quad C_{kj}e_r=\begin{cases}
	\eta \mathbf{e_k},&\text{ if } r=j\\
	-e,&\text{ if } 1\leq r\neq j\leq n
\end{cases}.\]
	Clearly, $B_{ij},C_{kj}$ are extreme contractions of $\mathcal{L}(\ell_1^n,\ell_1^m).$  Moreover, $\{Se_l,B_{ij}e_l\}$ and $\{Se_l,C_{ij}e_l\}$ are linearly independent. So using Lemma \ref{lem-md}, it is easy to observe that 
	$$\|T(B_{ij})\|=\|T(C_{ij})\|=\|T\|,$$ and by  Lemma \ref{lem-01}, $\frac{1}{\|T\|}T(B_{ij}),\frac{1}{\|T\|}T(C_{kj})$ are extreme contractions. Furthermore, since  $\{B_{ij}e_j, -C_{ij}e_j\}$ is linearly independent, and $B_{ij}e_l=-C_{ij}e_l$ for $j\neq l,$ by Lemma \ref{lem-md}, there exists $1\leq r\leq n,$ such that 
	either  $$\Big\{\frac{1}{\|T\|}T(B_{ij})e_r,\frac{1}{\|T\|}T(C_{kj})e_r\Big\} \text{ is linearly independent}$$  or $$\frac{1}{\|T\|}T(B_{ij})e_r=\frac{1}{\|T\|}T(C_{kj})e_r.$$ In each case, $$\|\frac{1}{\|T\|}T(B_{ij})e_r+\frac{1}{\|T\|}T(C_{kj})e_r\|_1=2.$$ Now, 
	\[2\|T\|\geq \|T(B_{ij})+T(C_{kj})\|\geq\|T(B_{ij})e_r+T(C_{kj})e_r\|_1=2\|T\|,\]
	which proves that $\frac{1}{2}\|T(B_{ij}+C_{kj})\|=\|T\|.$	\\
	
	(iii) If $|\alpha|=|\beta|,$ then since $B=|\alpha|(\xi E_{ij}+\eta E_{kj})$ for some $\xi ,\eta \in \mathbb{T},$ the proof follows from (ii). 
	Without loss of generality, assume that $|\alpha|>|\beta|.$ Then 
	\[B=\Big(|\alpha|-|\beta|\Big)\frac{\alpha}{|\alpha|}E_{ij}+|\beta |\Big(\frac{\beta}{|\beta|}E_{kj}+\frac{\alpha}{|\alpha|}E_{ij}\Big).\]
	Note that $$\Big(|\alpha|-|\beta|\Big)\frac{\alpha}{|\alpha|}E_{ij}\parallel |\beta |\Big(\frac{\beta}{|\beta|}E_{kj}+\frac{\alpha}{|\alpha|}E_{ij}\Big).$$ Therefore, $\Big(|\alpha|-|\beta|\Big)\frac{\alpha}{|\alpha|}T(E_{ij})\parallel |\beta |T\Big(\frac{\beta}{|\beta|}E_{kj}+\frac{\alpha}{|\alpha|}E_{ij}\Big).$ So there exists $\lambda\in \mathbb{T}$ such that 
	\begin{eqnarray*}
			&& \|\Big(|\alpha|-|\beta|\Big)\frac{\alpha}{|\alpha|}T(E_{ij})+\lambda |\beta |T\Big(\frac{\beta}{|\beta|}E_{kj}+\frac{\alpha}{|\alpha|}E_{ij}\Big)\|\\
			&=&\|\Big(|\alpha|-|\beta|\Big)\frac{\alpha}{|\alpha|}T(E_{ij})\|+|\beta|\|T\Big(\frac{\beta}{|\beta|}E_{kj}+\frac{\alpha}{|\alpha|}E_{ij}\Big)\|\\
			& =& \Big(|\alpha|-|\beta|\Big)\|T\|+2|\beta|\|T\|, \quad (\text{by (i), (ii)})\\
			&=& \Big(|\alpha|+|\beta|\Big)\|T\|.
	\end{eqnarray*}
	If $\lambda\neq 1,$ then the above equality does not hold, since for $\lambda \neq 1,$
	\begin{eqnarray*}
			&& \|\Big(|\alpha|-|\beta|\Big)\frac{\alpha}{|\alpha|}T(E_{ij})+\lambda |\beta|T\Big(\frac{\beta}{|\beta|}E_{kj}+\frac{\alpha}{|\alpha|}E_{ij}\Big)\|\\
			&=&\|\Big(|\alpha|+(\lambda-1)|\beta|\Big)\frac{\alpha}{|\alpha|}T(E_{ij})+\lambda \beta T(E_{kj})\|\\
			& \leq& \Big||\alpha|+(\lambda-1)|\beta|\Big|\|T(E_{ij})\|+|\beta|\|T(E_{kj})\|   \\
			&=& \Big(\Big||\alpha|+(\lambda-1)|\beta|\Big|+|\beta|\Big)\|T\|\\
			&<& \Big(|\alpha|+|\beta|\Big)\|T\|,
	\end{eqnarray*}
	a contradiction. Therefore, $\lambda=1$ and so
	\begin{eqnarray*}
		&& \Big(|\alpha|+|\beta|\Big)\|T\|\\
		&=&\|\Big(|\alpha|-|\beta|\Big)\frac{\alpha}{|\alpha|}T(E_{ij})+ |\beta|T\Big(\frac{\beta}{|\beta|}E_{kj}+\frac{\alpha}{|\alpha|}E_{ij}\Big)\|\\
		&=&\|T(B)\|.
	\end{eqnarray*}
	\end{proof}

\begin{lemma}\label{lem-05}
Suppose $S=\sum_{r=1}^i\alpha_rE_{rj}\in \mathcal{L}(\ell_1^n,\ell_1^m), ~(1\leq i\leq m,~1\leq j\leq n),$ where $\alpha_1,\ldots, \alpha_i$ are non-zero scalars. 	Then  
	$$\|T(S)\|=\Big(\sum_{r=1}^i|\alpha_r|\Big)\|T\|.$$
\end{lemma}
\begin{proof}
	We prove the lemma using induction on $i.$ 
	For $i=1,2,$ the result follows from Lemma \ref{lem-04} (i), and (iii).
	Suppose the lemma is true for $2\leq k\leq i-1,$ that is, if $S=\sum_{r=1}^{k}\alpha_rE_{rj},$ where $\alpha_1,\ldots, \alpha_{k}$ are non-zero scalars. 	Then  
	$\|T(S)\|=\Big(\sum_{r=1}^{k}|\alpha_r|\Big)\|T\|.$\\
	 Now, consider  $S=\sum_{r=1}^i\alpha_rE_{rj},$ where $i\geq 3.$ Let 
	 \begin{equation}\label{eq-09}
	 	A=\sum_{r=1}^{i-2}\alpha_{r}E_{rj}+\frac{\alpha_{i-1}}{2}E_{(i-1)j} \text{ and } B=\frac{\alpha_{i-1}}{2}E_{(i-1)j}+\alpha_iE_{ij}.
	 	\end{equation} 
 	Then $S=A+B.$
	 Since $$Ae_j=(\alpha_1,\ldots,\alpha_{i-2},\frac{\alpha_{i-1}}{2},0,\ldots,0),\quad  Be_j=(0,\ldots,0,\frac{\alpha_{i-1}}{2},\underbrace{\alpha_i}_{i\text{-th}},0,\ldots,0),$$ so $Ae_j\parallel Be_j.$ Moreover, $e_j\in M_{A}\cap M_B,$ which implies that $A\parallel B,$ that is, 
	 \begin{equation}\label{eq-05}
	 T(A)\parallel T(B).
	 \end{equation}
	 By induction hypothesis, 
	 \begin{equation}\label{eq-06}
	 \|T(A)\|=\Big(\sum_{r=1}^{i-2}|\alpha_{r}|+|\frac{\alpha_{i-1}}{2}|\Big)\|T\|.
	 \end{equation} 
 On the other hand, from Lemma \ref{lem-04} (iii), it follows that
 \begin{equation}\label{eq-07}
 \|T(B)\|=\Big(\frac{|\alpha_{i-1}|}{2}+|\alpha_i|\Big)\|T\|.
  \end{equation}
Now from \eqref{eq-05}, we get $\mu\in\mathbb{T}$ such that
\begin{equation}\label{eq-08}
 \|T(A)+\mu T(B)\|=\|T(A)\|+\|T(B)\|=\Big(\sum_{r=1}^i|\alpha_r|\Big)\|T\|,
 \end{equation} 
where the last equality follows from \eqref{eq-06} and \eqref{eq-07}.  From \eqref{eq-09}, we get
\begin{eqnarray*}
\|T(A)+\mu T(B)\|&=&\|\sum_{r=1}^{i-2}\alpha_{r}T(E_{rj})+\frac{\alpha_{i-1}}{2}(1+\mu)T(E_{(i-1)j})+\mu\alpha_iT(E_{ij})\|\\
&\leq& \Big(\sum_{r=1}^{i-2}|\alpha_{r}|+|\frac{\alpha_{i-1}}{2}(1+\mu)|+|\alpha_i|\Big)\|T\|\\
&<&\Big(\sum_{r=1}^i|\alpha_r|\Big)\|T\|, \quad {\text{if }\mu \neq 1}.
\end{eqnarray*}
Therefore in \eqref{eq-08},  $\mu=1.$ This proves that 
\[\|T(S)\|=\|T(A)+T(B)\|=\Big(\sum_{r=1}^i|\alpha_r|\Big)\|T\|.\]
\end{proof}

Finally, we are ready to prove the implication $(3)\Rightarrow (4),$ for $p=1.$ 
\begin{proof}
 Suppose $A\in\mathcal{L}(\ell_1^n,\ell_1^m)$ and $\|A\|=1.$ Let $A$ be smooth.   We first show that $\|T(A)\|=\|T\|.$ By Lemma \ref{lem-06}, there exists $1\leq j\leq n$ such that  $\|Ae_j\|_1=1,$ and  $\|Ae_i\|_1<1$ for all $1\leq i~(\neq j)\leq n.$ Moreover, $Ae_j=(c_{1j},c_{2j},\ldots, c_{mj}),$ where $c_{kj}\neq 0$ for all $1\leq k\leq m.$
Define
\[Be_i=\begin{cases}
	Ae_j,&\text{ if } i=j\\
	0,&\text{ if } 1\leq i\neq j\leq n
	\end{cases},\quad
De_i=\begin{cases}
	Ae_j,&\text{ if } i=j\\
	-Ae_i,&\text{ if } 1\leq i\neq j\leq n
	\end{cases}.\]
 Then $B=\frac{1}{2}(A+D).$
Since $B=\sum_{i=1}^mc_{ij}E_{ij},$ by Lemma \ref{lem-05}, $$\|T(B)\|=\Big(\sum_{i=1}^m|c_{ij}|\Big)\|T\|=\|Ae_j\|_1\|T\|=\|T\|.$$
Therefore, 
\begin{eqnarray*}
\|T\|&=&\|T(B)\|\\
&=&\frac{1}{2}\|T(A)+T(D)\|\\
&\leq &	\frac{1}{2}(\|T\|\|A\|+\|T\|\|D\|)\\
&=&\|T\|,
\end{eqnarray*}
which implies that $\|T(A)\|=\|T\|.$ Now, since $T$ attains norm at each smooth unit vector of $\mathcal{L}(\ell_1^n,\ell_1^m),$ and the smooth vectors are dense in $\mathcal{L}(\ell_1^n,\ell_1^m),$ (see \cite[P. 171]{H}) so $T$ attains norm at each unit vector of $\mathcal{L}(\ell_1^n,\ell_1^m).$
Therefore, $T$ is a scalar multiple of an isometry.
\end{proof}

\subsection{Proof of Theorem \ref{th-01} for $p=\infty$}\label{sub-infty}
Suppose $T:\mathcal{L}(\ell_\infty^n,\ell_\infty^m)\to \mathcal{L}(\ell_\infty^n,\ell_\infty^m)$ is a non-zero linear map.   Define $\tilde{T}:\mathcal{L}(\ell_1^m,\ell_1^n)\to \mathcal{L}(\ell_1^m,\ell_1^n)$ as follows:
\[\tilde{T}(A):=(T(A^*))^*.\]
Then it is easy to observe that $T$ preserves parallel (resp. TEA) pairs if and only if $\tilde{T}$ preserves parallel (resp. TEA) pairs. Indeed, assume that $T$ preserves parallel  pairs, and  
 $A,B\in \mathcal{L}(\ell_1^m,\ell_1^n)$ are such that $A\parallel B.$ Then 
	\[A^*\parallel B^*\Rightarrow T(A^*)\parallel T(B^*)\Rightarrow (T(A^*))^*\parallel (T(B^*))^*\Rightarrow \tilde{T}(A)\parallel \tilde{T}(B).\]
The proofs of the other parts hold similarly. 	From this observation we immediately get the proof of Theorem \ref{th-01} for $p=\infty.$ For the sake of completeness, we include a proof here.
\begin{proof}
$(1) \Rightarrow (2)$. Since $T$ preserves TEA pairs, so does $\tilde{T}.$ By Theorem \ref{th-01} for $p=1,$ $\tilde{T}$	preserves parallel pairs and $rank(\tilde{T})>1$. Thus, $T$ preserves parallel pairs. Clearly, $rank(T)>1.$ \\

 $(2)\Rightarrow (3).$ From $(2)$ it follows that  $\tilde{T}$	preserves parallel pairs and $rank(\tilde{T})>1$.
So by Theorem \ref{th-01} for $p=1,$ $\tilde{T}$ is invertible. Therefore, $T$ is invertible.\\

 $(3)\Rightarrow (4).$ From $(3)$ it is easy to observe that  $\tilde{T}$	preserves parallel pairs and $\tilde{T}$ is invertible. So by Theorem \ref{th-01} for $p=1,$ $\tilde{T}$ is a scalar multiple of an isometry. Let $A\in \mathcal{L}(\ell_\infty^n,\ell_\infty^m)$ and $\|A\|=1.$ Then 
\[\|T(A)\|=\|(T(A))^*\|=\|\tilde{T}(A^*)\|=\|\tilde{T}\|,\]
and so $T$ attains norm at each unit vector of $ \mathcal{L}(\ell_\infty^n,\ell_\infty^m).$ Therefore $T$ is a scalar multiple of an isometry.\\

 $(4)\Rightarrow (1)$  is trivial.
\end{proof}

\section{Anknowledgement}
The author would like to thank DST, Govt. of India for partial financial support in the form of INSPIRE Faculty Fellowship (DST/INSPIRE/04/2022/001207).

\bibliographystyle{amsplain}

\end{document}